\documentclass{amsart}
\theoremstyle{plain}
\newtheorem{thm}{Theorem}
\theoremstyle{remark}
\newcommand{\Z}{\mathbb{Z}}

\newcommand{\cond}{\,\vert\,}
\DeclareMathOperator{\Fix}{Fix}
\DeclareMathOperator{\rank}{rank}
\DeclareMathOperator{\codim}{codim}

\begin{document}
\title[The Hopf conjecture for manifolds with low cohomogeneity]{The Hopf
conjecture for manifolds with low cohomogeneity or high symmetry rank}
\author{Thomas P\smash{\"u}ttmann}
\address{Ruhr-Universit\"at Bochum\\Fakult\"at f\"ur Mathematik\\
         D-44780 Bochum\\Germany}
\email{puttmann@math.ruhr-uni-bochum.de}
\author{Catherine Searle}
\address{Instituto de Matematicas\\Unidad Cuernavaca-UNAM\\
Apartado Postal 273-3\\ Admon. 3\\
Cuernavaca\\Morelos, 62251\\Mexico}
\email{csearle@matcuer.unam.mx}
\date{May 15, 2000}
\thanks{This paper has its roots in the nonnegative curvature seminar
at the University of Pennsylvania in the academic year 1999--2000.
The first named author would like to thank the University of Pennsylvania
for their hospitality and A.~Rigas, B.~Wilking, W.~Ziller for advice
and many interesting discussions during this time. The second named
author was supported in part by a grant from CONACYT project number 28491-E}
\subjclass{53C20}

\begin{abstract}
  We prove that the Euler characteristic of an even-dimensional compact
  manifold with positive (nonnegative) sectional curvature is positive
  (nonnegative) provided that the manifold admits an isometric action
  of a compact Lie group $G$ with principal isotropy group $H$ and
  cohomogeneity $k$ such that $k - (\rank G - \rank H)\le 5$.
  Moreover, we prove that the Euler characteristic of a compact Riemannian
  manifold $M^{4l+4}$ or $M^{4l+2}$ with positive sectional curvature is
  positive if $M$ admits an effective isometric action of a torus $T^l$,
  i.e., if the symmetry rank of $M$ is $\ge l$.
\end{abstract}

\maketitle

The Gauss-Bonnet theorem states that the Euler characteristic
of a closed surface $M$ is determined by its total curvature:
$\chi(M) =2\pi\int_M K$. In particular, if the curvature is positive
(nonnegative), the Euler characteristic of the surface is positive
(nonnegative).
H.~Hopf \cite{hopf} generalized in 1925 the Gauss-Bonnet theorem to
even-dimensional hypersurfaces of Euclidean space and posed in the early
1930's (according to Berger \cite{berger}) the question whether a compact
even-dimensional manifold which admits a metric of positive (nonnegative)
sectional curvature must have positive (nonnegative) Euler characteristic.

Indications that the Hopf conjecture should be true came from the
generalizations of the Gauss-Bonnet theorem: Fenchel \cite{fenchel}
and Allendoerfer \cite{allendoerfer} proved in 1940 independently a
Gauss-Bonnet formula for submanifolds of Euclidean space with arbitrary
codimension. Three years later Allendoerfer and Weil \cite{weil}
(using E.~Cartan's result that any Riemannian manifold can locally be
embedded into Euclidean space) established the theorem in its final
intrinsic version:
For any even-dimensional manifold the Euler characteristic can be obtained
by integrating a function derived from the curvature tensor, the so called
Gauss-Bonnet integrand.
Chern \cite{chernone} gave the first intrinsic proof of this theorem in 1944.

After this, many attempts were made to settle the stronger algebraic
Hopf conjecture: A curvature tensor with positive (nonnegative)
sectional curvature yields a positive (nonnegative) Gauss-Bonnet integrand.
Milnor (unpublished, see \cite{cherntwo}) actually proved the
algebraic Hopf conjecture in dimension\,$4$, but finally in 1976
Geroch \cite{geroch} found curvature tensors with positive sectional
curvature in all even dimensions $\ge 6$ that do not provide a positive
Gauss-Bonnet integrand.

A different approach to the Hopf conjecture is to consider first 
Riemannian manifolds that have a certain amount of symmetry.
Hopf himself and Samelson \cite{samelson} proved in 1941 that the Euler
characteristic of every compact homogeneous space $G/H$ is nonnegative
and positive if and only if $\rank G = \rank H$ holds.
The key observations in their proof are that a regular element in
the compact Lie group $G$ has at most finitely many fixed points in
the homogeneous space $G/H$ and that each of these fixed points has
fixed point index\,$1$.
In 1972 Wallach \cite{wallach} then showed that for any even-dimensional
homogeneous space of positive sectional curvature one actually has
$\rank G = \rank H$. Therefore, the Hopf conjecture is true for
homogeneous spaces. Recently, Podest\`{a} and Verdiani \cite{podesta}
proved among other things that the Hopf conjecture also holds for
cohomogeneity one manifolds. We show that much weaker symmetry assumptions
are sufficient.

\begin{thm}
\label{thm:main}
  Let $M$ be a compact even-dimensional Riemannian manifold with positive
  (nonnegative) sectional curvature. Let $G\times M\to M$ be an isometric
  action of a compact Lie group $G$ with principal isotropy group $H$ and
  cohomogeneity $k$. If
  \begin{gather*}
    k - (\rank G - \rank H) \le 5
  \end{gather*}
  then $M$ has positive (nonnegative) Euler characteristic.
\end{thm} 

\begin{proof}
  If $M$ is nonorientable, then the action of $G$ can be lifted to an
  action by orientation preserving isometries
  (see \cite[Corollary I.9.4]{bredon}) on the orientable double covering
  space of $M$. We can therefore assume that $M$ is orientable.
  
  We consider the fixed point set
  $M^T = \{p\in M\cond \psi(p) = p\text{ for all $\psi\in T$}\}$
  of a maximal torus $T$ of $G$. Note that $M^T$ is equal
  to the fixed point set of a generating element $\psi\in T$, i.e.,
  of an element $\psi$ with $\overline{\{\psi^m\cond m\in \Z\}} = T\}$.
  If the fixed point set $M^T$ is empty then there exists a Killing
  field without zeros. This implies that $M$ cannot have positive 
  sectional curvature (by Berger's theorem, see e.g. \cite{wallach})
  and that the Euler characteristic is zero. We can therefore assume
  that $M^T$ is nonempty. Now each of the finitely many components of $M^T$
  is a totally geodesic submanifold of $M$ with even codimension and
  the Euler characteristic of $M$ is the sum of the Euler characteristics
  of the components (see \cite[Chapter II]{kobayashi}).
  By Theorem\,IV.5.3 of \cite{bredon} each component $N$ of $M^T$ satisfies
  \begin{gather*}
    \dim N \le k - (\rank G - \rank H) \le 5.
  \end{gather*}
  Since $N$ is even-dimensional and the Hopf conjecture holds in dimensions
  $2$ and $4$ we are done.
\end{proof}

Note that $k - (\rank G - \rank H) \le \dim M - 2\rank G$
(see \cite[Corollary\,IV.5.4]{bredon}) if the action of $G$ is effective.
Hence we get as a special case of Theorem\,\ref{thm:main} that any
compact even-dimensional Riemannian manifold $M^{2l+4}$ with positive
(nonnegative) sectional curvature has positive (nonnegative) Euler
characteristic if $M^{2l+4}$ admits an effective isometric torus action
$T^l\times M\to M$. Using a result from \cite{searle} we can improve
this result in the case of positive sectional curvature.
\begin{thm}
\label{t:symrank}
  Let $M^{4l+2}$ or $M^{4l+4}$ be a Riemannian manifold with
  positive sectional curvature that admits an almost effective
  isometric $T^l$-action. Then for any $T^1\subset T^l$
  the Euler characteristics of all the components of the fixed
  point set $\Fix(M;T^1)$ are positive. In particular, $\chi(M) > 0$.
\end{thm}
\begin{proof}
  As above we can assume that $M$ is orientable in order to have
  even-dimensional fixed point sets.
  The proof is done by induction. For $l=0$ note that the Hopf conjecture
  is true in dimensions $2$ and $4$. For the induction step consider
  $M^{4l+6}$ or $M^{4l+8}$ with an almost effective $T^{l+1}$-action.
  Consider any circle $T^1\subset T^{l+1}$ and any component $N$
  of its fixed point set in $M$. We will show that $\chi(N)>0$.
  Choose an $\tilde T^1\subset T^{l+1}$ such that
  $N\subset \Fix(M;\tilde T^1)$ and such that the component $\tilde N$
  of $\Fix(M;\tilde T^1)$ that contains $N$ has maximal dimension.
  It follows from the slice theorem and from the representation theory of
  tori that $T^l = T^{l+1}/\tilde T^1$ acts almost effectively on $\tilde N$.
  If $\codim \tilde N \ge 4$ we know from the induction assumption
  that in particular $N$ as a component of $\Fix(\tilde N;T^1)$ has
  positive Euler characteristic and hence we are done.
  In the case where $\codim \tilde N = 2$ we know from \cite{searle}
  that $M$ is differentiably covered by a sphere or a complex projective
  space. From results of Bredon \cite[Chapter III and VII]{bredon}
  it follows that all the components of the fixed point set of any circle
  action on $M$ have positive Euler characteristic. Thus in particular
  $N$ has positive Euler characteristic.
\end{proof}

After this paper was accepted for publication we have been informed
that Xiaochun Rong obtained Theorem\,\ref{t:symrank} independently
(see \cite{rong}). In his paper he gives many more results on the
topology of positively curved manifolds with high symmetry rank.

%
% ************ bibliography ************
%

%

\nocite{*}
\begin{thebibliography}{AW}
%
\bibitem[A]{allendoerfer}
C.\,B.~Allendoerfer, {\em The Euler number of a Riemannian manifold},
Amer.\ J.\ Math. {\bf 62} (1940), 243--248.
%
\bibitem[AW]{weil}
C.\,B.~Allendoerfer, A.~Weil, {\em The Gauss-Bonnet theorem for Riemannian
polyhedra}, Trans.\ Amer.\ Math.\ Soc. {\bf 53} (1943), 101--129.
%
\bibitem[Be]{berger}
M.~Berger, {\em Riemannian geometry during the second half of the twentieth
century}, University Lecture Series, 17,  American Mathematical Society,
Providence, RI, 1998.
%
\bibitem[Br]{bredon}
G.\,E.~Bredon, {\em Introduction to compact transformation groups},
Pure and Applied Mathematics, Vol. 46, Academic Press, New York-London 1972.
%
\bibitem[C1]{chernone}
S.\,S.~Chern, {\em A simple intrinsic proof of the Gauss-Bonnet formula for
closed Riemannian manifolds}, Ann.\ of\ Math.(2) {\bf 45} (1944), 747--752.
%
\bibitem[C2]{cherntwo}
S.\,S.~Chern, {\em On curvature and characteristic classes of a Riemann
manifold}, Abh.\ Math.\ Sem.\ Univ.\ Hamburg {\bf 20} (1955), 117--126.
%
\bibitem[F]{fenchel}
W.~Fenchel, {\em On total curvatures of Riemannian manifolds: I.},
J.\ London\ Math.\ Soc. {\bf 15} (1940), 15--22. 
%
\bibitem[G]{geroch}
R.~Geroch, {\em Positive sectional curvature does not imply positive
Gauss-Bonnet integrand}, Proc.\ Amer.\ Math.\ Soc. {\bf 54} (1976),
267--270.
%
\bibitem[GS]{searle}
K.~Grove, C.~Searle, {\em Positively curved manifolds with maximal
symmetry-rank}, J.\ Pure Appl.\ Algebra {\bf 91} (1994), 137--142.
%
\bibitem[H]{hopf}
H.~Hopf, {\em {\"U}ber die Curvatura integra geschlossener Hyperfl{\"a}chen},
Math.\ Ann. {\bf 95} (1925), 340--367.
%
\bibitem[HS]{samelson}
H.~Hopf, H.~Samelson, {\em Ein Satz {\"u}ber die Wirkungsr{\"a}ume geschlossener
Liescher Gruppen}, Comment.\ Math.\ Helv. {\bf 13} (1941), 240--251. 
%
\bibitem[K]{kobayashi}
S.~Kobayashi, {\em Transformation groups in differential geometry},
Ergebnisse der Mathematik und ihrer Grenzgebiete, Band 70, Springer,
Berlin-Heidelberg-New York 1972.
%
\bibitem[PV]{podesta}
F.~Podest\`{a}, L.~Verdiani, {\em Totally geodesic orbits of isometries},
Ann.\ Global Anal.\ Geom. {\bf 16} (1998), 413--418.
%
\bibitem[R]{rong}
X.~Rong, {\em Positively curved manifolds with almost maximal symmetry rank},
to appear.
%
\bibitem[W]{wallach}
N.\,R.~Wallach, {\em Compact homogeneous Riemannian manifolds with strictly
positive sectional curvature},
Ann.\ of Math. {\bf 96} (1972), 277--295.
%
\end{thebibliography}
\end{document}